\documentclass[twoside,11pt]{article}

\usepackage[preprint]{jmlr2e}

\usepackage[
    hmargin=1.25in, 
    vmargin=1.1in
]{geometry} 

\usepackage{amsmath}%
\usepackage{graphicx}%
\usepackage{multirow}%
\usepackage{amsmath,amssymb,amsfonts}%
\usepackage{thm-restate}
\usepackage{mathrsfs}%
\usepackage[title]{appendix}%
\usepackage{xcolor}%
\usepackage{textcomp}%
\usepackage{manyfoot}%
\usepackage{booktabs}%
\usepackage{algorithm}%
\usepackage{algorithmicx}%
\usepackage{algpseudocode}%
\usepackage{listings}%
\usepackage{cleveref}%

\usepackage{mathtools}
\usepackage{bm}
\usepackage{mathrsfs}
\usepackage{xparse}
\usepackage{nicefrac}

\def\setC{{\mathcal{C}}}

\def\setH{{\mathcal{H}}}

\def\setX{{\mathcal{X}}}
\def\setY{{\mathcal{Y}}}

\DeclareMathOperator*{\argmin}{arg\,min}

\newcommand{\nleft}{\mathopen{}\mathclose\bgroup\left}
\newcommand{\nright}{\aftergroup\egroup\right}

\NewDocumentCommand{\lin}{sm}{\IfBooleanTF{#1}{\langle#2\rangle}{\nleft\langle#2\nright\rangle}}
\NewDocumentCommand{\ceil}{sm}{\IfBooleanTF{#1}{\lceil#2\rceil}{\nleft\lceil#2\nright\rceil}}
\NewDocumentCommand{\floor}{sm}{\IfBooleanTF{#1}{\lfloor#2\rfloor}{\nleft\lfloor#2\nright\rfloor}}
\NewDocumentCommand{\norm}{sm}{\IfBooleanTF{#1}{\Vert#2\Vert}{\nleft\Vert#2\nright\Vert}}
\NewDocumentCommand{\abs}{sm}{\IfBooleanTF{#1}{\vert#2\vert}{\nleft\vert#2\nright\vert}}
\NewDocumentCommand{\paren}{sm}{\IfBooleanTF{#1}{(#2)}{\nleft(#2\nright)}}
\NewDocumentCommand{\parens}{sm}{\IfBooleanTF{#1}{(#2)}{\nleft(#2\nright)}}
\NewDocumentCommand{\braces}{sm}{\IfBooleanTF{#1}{\{#2\}}{\nleft\{#2\nright\}}}
\NewDocumentCommand{\brackets}{sm}{\IfBooleanTF{#1}{[#2]}{\nleft[#2\nright]}}

\newcommand{\R}{\mathbb{R}}

\DeclareMathOperator*{\E}{\mathbb{E}}
\DeclareMathOperator*{\V}{\mathrm{Var}}
\let\Pr\relax
\DeclareMathOperator*{\Pr}{\mathbb{P}}

\NewDocumentCommand{\KL}{smm}{%
	\mathrm{KL}{%
		\IfBooleanTF{#1}{%
			[#2\;\Vert\;#3]%
		}{%
			\nleft[#2\;\middle\Vert\;#3\nright]%
		}%
	}%
}
\NewDocumentCommand{\Expect}{som}{%
	\E%
	\IfNoValueTF{#2}{}{_{#2}}%
	\IfBooleanTF{#1}{[#3]}{\nleft[#3\nright]}%
}
\NewDocumentCommand{\Var}{som}{%
	\V%
	\IfNoValueTF{#2}{}{_{#2}}%
	\IfBooleanTF{#1}{[#3]}{\nleft[#3\nright]}%
}

\NewDocumentCommand{\Prob}{sm}{
	\Pr{%
		\IfBooleanTF{#1}{%
			[#2]%
		}{%
			\nleft[#2\nright]%
		}%
	}%
}
\DeclareDocumentCommand{\Normal}{g}{\mathop{\mathcal{N}}\IfNoValueF{#1}{\nleft(#1\nright)}}

\DeclareDocumentCommand{\bigO}{g}{\mathop{\mathcal{O}}\IfNoValueF{#1}{\nleft(#1\nright)}}

\newcommand{\aligns}[1]{\begin{align*}#1\end{align*}}
\newcommand{\alignn}[1]{\begin{align}#1\end{align}}

\def\m!{\mkern-2mu}
\def\n!{\mkern-1mu}
\def\p!{\mkern-.5mu}
\def\rn!{\mkern1mu}
\def\rp!{\mkern.5mu}

\newenvironment{textstylemath}{
    \let\olddisplaystyle\displaystyle%
    \let\displaystyle\textstyle%
}{%
    \let\displaystyle\olddisplaystyle%
    \ignorespacesafterend%
}

\let\nablasymbol\nabla
\RenewDocumentCommand{\nabla}{e{_^}}{%
    \nablasymbol
    \IfValueT{#1}{_{\mspace{-4mu}#1\,}}%
    \IfValueT{#2}{^{#2}}\n!%
}

\DeclareMathOperator{\dom}{dom}
\DeclareMathOperator{\interior}{int}

\DeclareMathOperator{\relint}{ri}

\newcommand*{\cC}{\mathcal{C}} 
 
\newcommand*{\cH}{\mathcal{H}}

\newcommand*{\cN}{\mathcal{N}}\newcommand*{\cX}{\mathcal{X}}

\makeatletter
\renewcommand{\paragraph}{\@startsiction{paragraph}{4}{\z@}{1.5ex plus
  0.5ex minus .2ex}{-0.5em}{\normalsize\bf}}
\makeatother
\newcommand{\removesomewhitespace}{}%
\newcommand{\slashfrac}[2]{#1/#2}

\ShortHeadings{Mirror Polyak}{Kunstner, D'Orazio, Portella, and Taylor}
\firstpageno{1}

\newcommand{\noqed}{\renewcommand{\BlackBox}{}}
\newcommand{\jmlrQED}{\BlackBox{}}

\AtBeginDocument{%
    \setlength{\abovedisplayskip}{8pt plus 4pt minus 1pt}%
    \setlength{\belowdisplayskip}{8pt plus 4pt minus 1pt}%
}

\begin{document}

\title{Mirror Polyak and a Primal-Dual Lifting}

\author{\name Frederik Kunstner \email frederik.kunstner@inria.fr \\
        \addr INRIA, École Normale Supérieure
        \AND
        \name Ryan D'Orazio \email ryan.dorazio@mila.quebec \\
        \addr MILA, Université de Montréal
        \AND
        \name Victor S. Portella \email v.portella@ufabc.edu.br \\
        \addr CMCC, Federal University of ABC
        \AND
        \name Adrien Taylor \email adrien.taylor@inria.fr \\
        \addr INRIA, École Normale Supérieure}

\maketitle

\begin{abstract}%
    First-order methods 
    typically require a specific step-size 
    that depends on the regularity conditions of the objective function,
    such as the smoothness, Lipschitz continuity, or strong convexity constants.
    The Polyak step-size is a classical alternative 
    for subgradient descent on convex functions
    that only uses knowledge of the optimal value of the objective function 
    and automatically adapts to the above-mentioned regimes.
    However, many optimization problems are 
    better described by non-Euclidean geometries
    and are more amenable to mirror descent.
    Extending this adaptivity to mirror descent is subtle.
    Some existing generalizations of the Polyak step-size 
    rely on norms instead of purely on relative geometry,
    excluding many of the use cases of mirror descent.
    In this work, we revisit a variant of the Polyak step-size based on Bregman projections
    due to~\citet{kiwiel1997bregman}, which we call \emph{mirror Polyak}.
    This method is known to converge asymptotically,
    but its convergence rate is not known.
    We show that mirror Polyak enjoys 
    guarantees similar to its Euclidean counterpart, 
    automatically adapting to relative notions of 
    smoothness, Lipschitz continuity, or strong convexity. 
    We then leverage mirror Polyak 
    to avoid having to know the optimal value 
    in some structured optimization problems
    such as regularized linear and logistic regression.
    We propose a lifted formulation based on convex duality 
    with optimal value exactly zero 
    and a natural mirror map given by the problem's structure. 
    Mirror Polyak applied to the lifted problem 
    enjoys the same worst-case guarantees 
    as the Polyak step-size in the original problem
    if we knew the optimal value.
\end{abstract}
\newcommand{\citeauthornolink}[1]{\begin{NoHyper}\citeauthor{#1}\end{NoHyper}}

\section{Introduction}

The Polyak step-size~\citep{polyak1969minimization}
is one of the simplest mechanisms to pick the step-size of 
(sub)gradient descent
to optimize a convex objective function ${f:\R^d \to \R}$.
Given knowledge of the optimal value~$f_\star \coloneqq \min_{x \in \R^d} f(x)$
and a subgradient~$g(x_t) \in \partial f(x_t)$ at an iterate \(x_t \in \R^d\), 
gradient descent with the Polyak step-size is defined as 
\alignn{
    \label{eq:euclidean_polyak_ss}
    x_{t+1} = x_t - \alpha_t g(x_t), 
    \quad \text{ where } \quad 
    \alpha_t = \frac{f(x_t) - f_{\star}}{\norm{g(x_t)}_2^2}.
} 
The convergence rate of the method 
automatically adapts to the regularity conditions of the problem,
such as smoothness, Lipschitzness, or strong convexity, 
trading the need to know those problem-specific constants with the need to know $f_\star$~\citep{hazan2019revisiting}. 

A limitation of the Polyak step-size is that it is inherently Euclidean.
Problems that are neither smooth, Lipschitz, nor strongly convex
are better handled with algorithms based on different geometries such as mirror descent~\citep{nemirovski1983problem,beck2003167}.
For example, the geometry induced by the Burg entropy 
is useful in Poisson inverse problems~\citep{bauschke2017descent,lu2018relatively},
while variants of Sinkhorn's algorithm 
are instances of mirror descent using the negative Shannon entropy~\citep{benamou2015iterative,mishchenko2019sinkhorn}. 
Mirror descent is particularly effective 
when there is a good match between 
the reference function and the geometry of the problem, 
described using relative variants of 
smoothness, strong convexity or Lipschitz continuity~\citep{birnbaum2011distributed,bauschke2017descent,lu2018relatively,lu2019relativecontinuity}.

This paper has two contributions. 
First, we analyze an extension of the Polyak step-size to mirror descent 
under relative regularity conditions.
Second, we present a lifted primal-dual gap construction 
that removes the need to know $f_{\star}$ for structured problems.
These two contributions are linked by the fact that the lifted problem 
may lose regularity properties such as Euclidean smoothness.
But it retains regularity properties with respect to a natural mirror map,
making it well-suited 
for a mirror descent extension of the Polyak~step-size. 

\paragraph{Mirror Polyak.}
We present a non-asymptotic analysis of a generalization 
of the Polyak step-size to mirror descent
based on Bregman projections due to~\citet{kiwiel1997bregman},
which we call \emph{mirror Polyak}.
We show that it enjoys almost identical\footnote{Up to a log-factor in the case of relatively strongly convex and Lipschitz functions.} 
adaptive convergence guarantees 
as the Euclidean Polyak step-size~(summarized in \cref{tbl:rates}), 
but under relative notions of smoothness, 
strong convexity and Lipschitzness~\citep{bauschke2017descent,lu2018relatively,lu2019relativecontinuity}. 
While the relatively smooth and relatively Lipschitz 
cases follow existing proof strategies, 
establishing convergence rates under relative strong convexity 
requires new arguments. 
In particular, step-size bounds or closed-form recursions 
on the distance to the optimum used in previous analyses 
do not work without assuming the mirror map 
is strongly convex relative to a norm.
These results support the view that Kiwiel's mirror Polyak is a more natural extension of the Polyak step-size 
to mirror descent, as other generalizations~\citep{you2022minimizing,dorazio2023stochastic}
require additional assumptions to be well-defined.

\paragraph{Lifting.}
We leverage mirror Polyak 
to remove the need to know $f_\star$ 
by instead assuming we know the structure of the problem.
A known but little-used trick is to optimize a lifted primal-dual gap, 
which has a known optimal value of $0$~\citep[][\S 1.2]{devanathan2024polyak}.
This approach is applicable to 
regularized linear or logistic regression problems,
but the lifted problem may no longer be smooth or strongly convex 
in the original Euclidean geometry.
We show that a natural choice of reference function for mirror descent 
has the same bound on the condition number %
in a relative geometry. 
This scheme allows mirror Polyak 
to achieve the same worst-case convergence guarantees 
as the classical Polyak step-size if we knew $f_\star$,
at the cost of doubling the per-iteration cost 
due to the evaluation of the gradient of the primal and dual problem.

\begin{table}[tbp]
\centering
\begin{tabular}{@{}lllll@{}}
    \toprule
    & $L$-smooth & $L$-smooth & $G$-Lipschitz & $G$-Lipschitz \\
     & 
    Convex & 
    $\mu$-strongly convex  & 
    Convex  & 
    $\mu$-strongly convex  
    \\
    \midrule
    \textbf{Mirror Polyak }
    & $\displaystyle \frac{L}{T}$
    $\,$ $(\bar x_T)$ 
    & $\displaystyle \paren{1-\frac{\mu}{L+\mu}}^T$
    $(x_T)$ 
    & $\displaystyle \frac{G}{\sqrt{T}}$
    $\,$ $(\bar x_T)$ 
    & $\displaystyle \frac{G^2}{\mu}\frac{\log T}{T}$ 
    $\,$ $(x^{\mathrm{best}}_T)$ 
    \\[.75em]
    \midrule
    \textbf{Euclidean Polyak}
    & $\displaystyle \frac{L}{T}$
    & $\displaystyle \paren{1-\frac{\mu}{L+\mu}}^T$ $\dagger$
    & $\displaystyle \frac{G}{\sqrt{T}}$
    & $\displaystyle \frac{G^2}{\mu} \frac{1}{T}$ \\[.75em]
    \bottomrule
\end{tabular}
\vspace{.5em}
\caption{%
Rates for \(T\) iterations of mirror Polyak
vs.\ best Euclidean rates~\citep{hazan2019revisiting}, 
omitting initial conditions. 
Rates hold for~$x_T$:~last,~$\bar x_T$:~average,~$x^{\mathrm{best}}_T$:~best iterate~(all rates hold~for~$x^{\mathrm{best}}_T$).
Smoothness, strong convexity, and Lipschitz continuity
are relative to a reference function.
$\dagger$:~\citet{hazan2019revisiting} 
get a worse rate of $(1-{\mu}/{2L})^T $ via~a~different~technique.
}
\vspace{-.5em}
\label{tbl:rates}
\end{table}

\paragraph{Notation.}
We study the minimization of a 
proper convex function~{$f \colon \setX \to \R$} on~{$\setX \subseteq \R^d$},
potentially constrained to a closed convex set~{$\setC \subseteq \R^d$}
containing a finite minimum~${f_{\star} \coloneqq f(x_{\star})}$ at~$x_{\star} \in \setX \cap \cC$.
We write~$g(x) \in \partial f(x)$ for a selection of a subgradient of~\(f\) at \(x\).
The mirror descent method %
relies on a convex reference function~$h \colon \setX \to \R$ of Legendre type~\citep[\S 26]{rockafellar1970convex},
and the Bregman~divergence $D$ induced by $h$,
\aligns{
    D(x, y) \coloneqq h(x) - h(y) - \lin{\nabla h(y), x - y}
    \qquad \forall x \in \setX, y \in \interior(\setX).
}
For any $x \in \interior(\setX)$, we denote its dual through~$h$ by~$\hat x \coloneqq \nabla h(x)$,
to write the mirror descent update 
and the three-point identity for Bregman divergences as
\aligns{
    \hat x_{t+1} = \hat x_t - \alpha_t g(x_t),
    &&
    D(x, z) = D(y, z) + \lin{\hat y - \hat z, x - y} + D(x, y).
}

\section{The Polyak step-size and mirror descent}

The expression for the Polyak step-size~\eqref{eq:euclidean_polyak_ss}
is sometimes taken as its \emph{definition}.
This view has motivated generalizations of the Polyak step-size to mirror descent~\citep{you2022polyaktypestepsizesmirror,you2022minimizing,dorazio2023stochastic}
that replace the Euclidean norm with a dual norm:
\alignn{
    \hat x_{t+1} = \hat x_t - \alpha_t g(x_t),
    &&
    \alpha_t = \frac{f(x_t) - f_{\star}}{\norm*{g(x_t)}_\ast^2}, \qquad\forall t \geq 0, 
    &&
    \text{ where } \hat x = \nabla h(x).
    \label{eq:norm-based-polyak}
}
Here, \(\norm{\cdot}_\ast\) is the dual 
of a reference norm~$\norm{\cdot}$.
Their analyses
rely on relationships between~$f$ and the norm~$\norm{\cdot}$,
assuming for example that~$h$ is strongly convex 
and $f$ has Lipschitz continuous gradients
with respect to the norm~$\norm{\cdot}$~\citep{dorazio2023stochastic, you2022minimizing}.
This assumption is stronger than necessary for mirror descent,
as the method can be defined and analyzed 
using only relationships between the objective~$f$
and the reference function~$h$,
without requiring any norm to be well-defined~\citep{bauschke2017descent,lu2018relatively}. 
These assumptions exclude applications where the 
reference function is not strongly convex with respect to any norm, 
such as the Burg entropy
for Poisson regression~\citep[][]{bauschke2017descent}
or log-determinant problems~\citep[][]{lu2018relatively}.
They also exclude problems which are not smooth in any norm, 
such as the dual problem for logistic regression 
that arises in the lifting described in~\Cref{sec:lifting}.

Instead, we show that 
a geometric perspective
gives
a norm-free generalization 
of the Polyak step-size to the mirror descent setting.
The Euclidean Polyak step-size
can be viewed 
as a projection onto a hyperplane that separates the current iterate and the optimum,
or equivalently as the minimization of 
an upper bound on the distance to the optimum, 
using the convexity of~$f$ 
and the knowledge of the optimum value~$f_{\star}$.
This view was already given by \citet{polyak1969minimization},
and is folklore in the community~\citep[][]{gower2025analysis}.
For simplicity, we assume \(\dom f = \setC = \R^d\) in this section.
\begin{proposition}[Euclidean Polyak step-size]
    \label{prop:euclidean-polyak-interpretation}
The Polyak step-size~\eqref{eq:euclidean_polyak_ss}
results from the projection of the iterate~$x_t$
onto the following hyperplane 
that separates~$x_t$~from~$x_{\star}$,
\aligns{
    x_{t+1} = \arg\min_{y \in \R^d} 
    \tfrac{1}{2}\norm{y-x_t}^2 : 
    f(x_t) + \lin{g(x_t), y - x_t} = f_{\star}.
}
Alternatively, 
the update can be viewed as~$x_{t+1} = x_t - \alpha_t g(x_t)$ 
where~$\alpha_t$ minimizes the bound
\aligns{
    \tfrac{1}{2}\norm{x_{t+1}-x_{\star}}^2
    \leq
    \min_\alpha
    \tfrac{1}{2}\norm{x_{t}-x_{\star}}^2
    - \alpha \paren{f(x_t) - f_{\star}}
    + \alpha^2 \tfrac{1}{2}\norm{g(x_t)}^2.
}

\end{proposition}
\removesomewhitespace{}
\begin{proof}
For the projection,
the KKT conditions give that there is~$\alpha_t \geq 0$ such that
\alignn{
    \label{eq:euclid_polyak_projection}
    x_{t+1} = x_t - \alpha_t g(x_t) 
    \quad \text{ and } \quad
    \lin{g(x_t), x_{t+1} - x_t} + \paren{f(x_t) - f_{\star}} = 0.
}
Together we get \(\alpha_t \norm{g(x_t)}^2 = {f(x_t) - f_{\star}}\), the Polyak step-size~\eqref{eq:euclidean_polyak_ss}.
Alternatively, note that
\aligns{
    \tfrac{1}{2}\norm{x_{t+1}-x_{\star}}^2
    &=
    \tfrac{1}{2}\norm{x_{t}-x_{\star}}^2
    + \alpha \lin{g(x_t), x_{\star} - x_{t}}
    + \alpha^2 \tfrac{1}{2}\norm{g(x_t)}^2,
    \\
    &\leq
    \tfrac{1}{2}\norm{x_{t}-x_{\star}}^2
    - \alpha \paren{f(x_t) - f_{\star}}
    + \alpha^2 \tfrac{1}{2}\norm{g(x_t)}^2,
}
by convexity.
Minimizing in \(\alpha\) gives
the Polyak step-size~\eqref{eq:euclidean_polyak_ss}.
\end{proof}

\paragraph{Mirror Polyak.} 
Based on this geometric view, 
the Polyak step-size can be generalized 
to a projection method, as proposed by \citet{kiwiel1997bregman},
which we call \emph{mirror Polyak}.
The idea is to replace the Euclidean projection in~\eqref{eq:euclid_polyak_projection} with a Bregman projection, yielding
\begin{equation}
    \label{eq:def_mirror_polyak}
    x_{t+1} \in  \argmin_{y \in \setX} D(y, x_t) 
   \quad \text{ such that } \quad 
   f(x_t) + \lin{g(x_t), y - x_t} = f_{\star}.
\end{equation}
This can be written as mirror descent with a Polyak-like step-size
that minimizes a bound on the Bregman divergence to the optimum.

\begin{proposition}[Mirror Polyak]
    \label{prop:mirror-polyak}
    Given $x_t \in \interior \cX$,
    Mirror Polyak \eqref{eq:def_mirror_polyak} 
    is a mirror descent step, 
    $\hat x_{t+1} = \hat x_{t+1}(\alpha_t) 
    \coloneqq \hat x_t - \alpha_t g(x_t)$,
    where the step-size \(\alpha_t \geq 0\) 
    minimizes
    \aligns{
        D(x_{\star}, x_{t+1}(\alpha_t)) 
        &\leq
        \min_\alpha 
        D(x_{\star}, x_t)
        - \alpha \paren{f(x_t) - f_{\star}}
        + D(x_t, x_{t+1}(\alpha)).
    }
\end{proposition}
\removesomewhitespace
\begin{proof}
    Since \(x_t \in \interior \cX\), the KKT conditions guarantee
    that there is $\alpha_t \geq 0$ such that
    \aligns{
        \hat x_{t+1} = \hat x_t - \alpha_t g(x_t)
        \quad \text{and} \quad 
        f(x_t) + \lin{g(x_{t}), x_{t+1} - x_t} = f_{\star}.
    }
    Once combined, we obtain the mirror Polyak step-size condition
    \alignn{
        \alpha_t \paren{f(x_t) - f_{\star}}
        =
        \lin{\hat x_t - \hat x_{t+1}, x_t - x_{t+1}}.
        \label{eq:relative-polyak-condition}
    }
    We refer to \(\alpha_t\) satisfying the above condition as the \emph{mirror Polyak step-size}.
    To obtain the upper bound in the divergence, 
    one can use the three-point identity and convexity to bound
    \aligns{
        D(x_{\star}, x_{t+1}(\alpha)) 
        &= 
        D(x_t, x_{t+1}(\alpha)) + \lin{\alpha g(x_t), x_{\star} - x_t} + D(x_{\star}, x_t),
        \\
        &\leq
        D(x_t, x_{t+1}(\alpha)) - \alpha \paren{f(x_t) - f_{\star}} + D(x_{\star}, x_t).
    }
    This bound is convex in $\alpha$ because $D(x_t, x_{t+1}(\alpha)) = D_{h^*}(\hat x_t - \alpha g(x_t), \hat x_t)$ where $D_{h^*}$ is the divergence induced by $h^*$, %
    and its minimum satisfies the mirror Polyak condition~\eqref{eq:relative-polyak-condition}.
\end{proof}
The method has also been studied 
as an instance of subgradient projectors~\citep{bauschke2003bregman} 
or firmly nonexpansive operators~\citep{cegielski2012iterative, bauschke2019convex}.
This generalization is not as well-known 
beyond the convex analysis community.
Recent works have instead focused on 
norm-based generalizations~\citep{you2022minimizing, dorazio2023stochastic}
and analyses under norm-based regularity assumptions~\citep{Gower2024a}.
We argue that mirror Polyak~(\cref{eq:def_mirror_polyak})
is the more natural generalization 
of the Polyak step-size to mirror descent,
preserving its original geometric intuition.
We show that, as in the Euclidean case, 
mirror Polyak automatically adapts to the degree of 
relative smoothness, Lipschitz continuity and strong convexity, 
without assuming $h$ to be strongly convex in~any~norm.

\subsection{Properties of mirror Polyak}

Before diving into the results, 
we discuss some properties of mirror Polyak,
including its computability, 
relation with the norm-based variant,
and handling of~constraints.

\paragraph{Computability.}
The mirror Polyak step-size is defined implicitly,
as it requires finding a step-size~$\alpha_t$ satisfying~\cref{eq:relative-polyak-condition}.
This equation has an explicit solution for quadratic~$h$.
But even for simple reference functions 
such as the negative entropy,~$\alpha_t$
can only be approximated~\citep{malitsky2025entropic}. 
Fortunately, finding $\alpha_t$ boils down to finding 
the minimum of a one-dimensional convex function, 
so root-finding procedures 
can find a solution to floating-point precision in a few iterations~\citep{kiwiel1998generalizedbp}.
Computing a step of mirror Polyak 
thus requires a few evaluations of~$\nabla h$ and $\nabla h^*$,
but only one evaluation of~$\nabla f$,
which is usually the most expensive operation.
For a typical regression problem 
with~$n$ samples in~$d$ dimensions,
computing~$\nabla f$ costs $O(nd)$ 
while computing~$\nabla h$ costs $O(d)$.

\paragraph{Comparison with norm-based variant.}
Even when the reference function~$h$ is strongly convex
and the norm-based Polyak step-size~\eqref{eq:norm-based-polyak} is well-defined,
the mirror Polyak step-size~\eqref{eq:relative-polyak-condition} is larger,
and may perform better, as shown in \Cref{sec:experiments}.
\begin{proposition}
    \label{prop:relative-polyak-larger}
    If \(h\) is 1-strongly convex
    with respect to a norm~\(\norm{\cdot}\),
    the mirror Polyak step-size $\alpha_t$ in~\cref{eq:relative-polyak-condition}
    is at least as large as the norm-based variant in~\cref{eq:norm-based-polyak}.
\end{proposition}
\removesomewhitespace
\begin{proof}
    Using the $1$-strong convexity of \(h\) in \(\norm{\cdot}\),
    we have~\citep[][Thm~2.1.10]{nesterov2018lectures}
    \aligns{
        \alpha_t (f(x_t) - f_{\star})
        &= \lin{\hat x_t - \hat x_{t+1}, x_t - x_{t+1}}
        \leq 
        \norm{\hat x_t - \hat x_{t+1}}_*^2
        = 
        \alpha_t^2  \lVert g(x_t) \rVert_*^2.
    }
    Rearranging yields 
    $\alpha_t \geq \paren{f(x_t) - f_{\star}}/{\norm{g(x_t)}_*^2}$,
    the norm-based Polyak step-size~\eqref{eq:norm-based-polyak}.
\end{proof}

\paragraph{Handling constraints.}
If~$f$ is restricted to a closed convex set~\(\setC \subseteq \R^d\),
\citet{kiwiel1997bregman} discusses two options.
We can first project onto the hyperplane,
then onto \(\setC\):
\aligns{
    x_{t+1} \in \argmin_{x \in \setX \cap \setC} D(x, x_{t+\nicefrac{1}{2}}),
    &&
    x_{t+\nicefrac{1}{2}} \in \argmin_{x \in \setX} D(x, x_t) :
    f(x_t) + \lin{g(x_t), x - x_t} = f_{\star}.
}
Another option is to project
onto the intersection of the hyperplane and the constraint:
\begin{equation}
    x_{t+1} \in \argmin_{x \in \setX \cap \setC} D(x, x_t) 
    \quad \text{ such that } \quad 
    f(x_t) + \lin{g(x_t), x - x_t} = f_{\star}.
    \label{eq:constrained_polyak_def}
\end{equation}
Even in the Euclidean case, 
both variants lead to different methods~\citep{kimu93}. 
Our results apply to both variants 
in the case of relatively Lipschitz functions,
but the rates for relatively smooth functions 
do not directly apply to the sequential projection.
We focus on the joint projection~\eqref{eq:constrained_polyak_def},
which has the following properties. 
\begin{lemma}[Properties of mirror Polyak]
    \label{prop:breg_stationarity}
    Assume \(x_0 \in \setC \cap \interior(\setX)\) 
    and let \(x_{t+1}\) be given as in~\eqref{eq:constrained_polyak_def}.   
    If
    \(x_{\star} \in \relint(\setC) \cap \interior(\setX) \), then there  is \(\alpha_t \geq 0\) for each $t \geq 0$ such that
    \begin{subequations}
    \begin{align}
        \label{eq:general-mirror-descent-step}
        \hat x_t - \alpha_t g(x_t) - \hat x_{t+1} &\in \cN_{\setC}(x_{t+1}),
        \quad \text{and} \quad 
        f(x_t) + \lin{g(x_t), x_{t+1} - x_t} = f_{\star},
        \\
        \label{eq:general-relative-polyak-condition}
        \alpha_t (f(x_t) - f_{\star}) 
        &\geq \lin{\hat x_t - \hat x_{t+1}, x_t - x_{t+1}} =  D(x_{t+1}, x_t) + D(x_t, x_{t+1}),
        \\
        \label{eq:breg_stationarity}
        D(x_{\star}, x_{t+1}) &\leq D(x_{\star}, x_t) - D(x_{t+1}, x_t).
    \end{align}
    \end{subequations}
    Here, $\cN_{\setC}(x)$ is the normal cone of \(\setC\) at $x \in \setC$,~$\cN_{\setC}(x) \coloneqq \{p \in \R^d \colon \lin{p, z - x} \leq 0 ~\forall z \in \setC\}$.
    
\end{lemma}
\removesomewhitespace
\begin{proof}
    Since $x_{\star} \in \relint(\setC) \cap \interior(\setX)$ 
    constraint qualification holds
    for the joint projection~\eqref{eq:constrained_polyak_def},
    see Appendix~\ref{apx:slater-condition} for details and possible relaxations of this assumption.
    Thus, 
    the existence of \(\alpha_t\) as in~\cref{eq:general-mirror-descent-step} 
    follows from the KKT conditions and since 
    \(x_{t+1} \in \interior \setX\) 
    as~$h$ is Legendre~\citep[Thm.~3.12]{bauschke1997}. 
    Equation~\eqref{eq:general-relative-polyak-condition} 
    follows from~\cref{eq:general-mirror-descent-step},
    \begin{equation*}
        \alpha_t(f(x_t) - f_{\star}) 
        = \lin{-\alpha_t g(x_t),x_{t+1} - x_{t}} 
        \geq \lin{\hat x_{t+1} - \hat x_t, x_{t+1} - x_t},
    \end{equation*}
    where the inequality is the normal cone condition~\eqref{eq:general-mirror-descent-step},
    and the equality to $D(x_{t+1}, x_t) + D(x_t, x_{t+1})$ 
    follows from the three-point identity of Bregman divergences.
    Finally, for~\cref{eq:breg_stationarity}, 
    combining the decrease of the Bregman divergence in\footnote{%
    Although this result assumes $\cC = \R^d$, 
    the inequality still holds with a very similar proof when $\cC \neq \R^d$.
    }
    Proposition~\ref{prop:mirror-polyak} 
    with \cref{eq:general-relative-polyak-condition}
    gives
    \begin{equation*}
        D(x_{\star}, x_{t+1}) \leq D(x_{\star}, x_t) -\alpha_t(f(x_t)-f(x_{\star}))  + D(x_{t}, x_{t+1}) 
        \leq 
        D(x_{\star}, x_t) - D(x_{t+1}, x_t).
        \tag*{\BlackBox}
    \end{equation*}
    \noqed
\end{proof}
All the results in the following sections 
assume the conditions of Lemma~\ref{prop:breg_stationarity}~hold.

\section{Convergence rates across function classes}
\label{sec:rates}

While convexity of $f$ is sufficient 
to establish the asymptotic convergence of the method,
getting a global non-asymptotic rate requires additional assumptions.
We consider the generalizations of smoothness and strong convexity 
of~\citet{birnbaum2011distributed,bauschke2017descent,lu2018relatively},
and the generalization of Lipschitz continuity
of~\citet{lu2019relativecontinuity},
using inner products instead of norms for the latter 
following \citet{zhoup0h20a}. 
These definitions reduce to their classical Euclidean 
counterparts when $h = \tfrac{1}{2}\norm{\cdot}_2^2$.
\begin{definition}[Relative geometry]
    We say that $f$ is 
    $L$-smooth, $\mu$-strongly convex, or $G$-Lipschitz continuous
    relative to $h$ 
    if, for all $x \in \interior(\setX)$ and $y \in \setX$,
    \aligns{
        &\text{$L$-smooth: }
        &
        f(y) &\leq f(x) + \lin{g(x), y - x} + L D(y,x), 
        \\
        &\text{$\mu$-strongly convex: }
        &
        f(y) &\geq f(x) + \lin{g(x), y - x} + \mu D(y,x), 
        \\
        &\text{$G$-Lipschitz continuous: }
        &
        \lin{g(x), x - y} &\leq G \sqrt{2 D(y, x)}.
    }
\end{definition}

\subsection{Convergence rate under relative smoothness}
\label{sec:rel-smoothness}
For a convex and relatively smooth function,
we recover the expected sublinear rate.

\begin{theorem}
    \label{thm:smooth-convex}
    If \(f\) is \(L\)-smooth relative to \(h\), 
    mirror Polyak guarantees
    \begin{equation*}
        f(\bar x_T) - f_{\star} \leq \frac{L D(x_{\star}, x_0)}{T}, 
        \quad \text{ for } \quad \bar x_T \coloneqq \frac{1}{T} \sum_{t=1}^T x_{t}. 
    \end{equation*}
\end{theorem}
\removesomewhitespace
\begin{proof}
    Using relative smoothness together with Lemma~\ref{prop:breg_stationarity},  we have
    \aligns{
        L D(x_{t+1}, x_t) 
        \geq f(x_{t+1}) - (f(x_t) + \lin{g(x_t), x_{t+1} - x_t}) 
       \stackrel{\smash{\eqref{eq:general-mirror-descent-step}}}{=} 
        f(x_{t+1}) - f_{\star}.
    }
    Combined with the contraction of the divergence to \(x_*\) in~\eqref{eq:breg_stationarity} 
    from Lemma~\ref{prop:breg_stationarity}, we get
    \aligns{
        D(x_{\star}, x_{t+1}) 
        \leq 
        D(x_{\star}, x_t) - \tfrac{1}{L} \paren{f(x_{t+1}) - f_{\star}}.
    }
    Iterating and using Jensen to bound \(f(\bar x_T)\) yields
    (this bound also holds for the best iterate)
    \aligns{
        f(\bar x_T) - f_{\star}
        \leq \frac{1}{T} \sum_{t=1}^T \paren{f(x_{t}) - f_{\star}}
        \leq \frac{L D(x_{\star}, x_0)}{T}.
        \tag*{\jmlrQED}
    }
    \noqed
\end{proof}

\vspace{-1.5em}

\subsection{Relative smoothness and strong convexity}
\label{sec:rel-smooth-strongly-convex}

In the relatively smooth and strongly convex case, we obtain a linear convergence rate. 
The proof is more involved than in the Euclidean case,
as we do not have a constant lower bound on the step-sizes~\(\alpha_t\)
unless $h$ is strongly convex with respect to a norm.

\begin{theorem}
    \label{thm:smooth-strongly-convex}
    If $f$ is $L$-smooth and $\mu$-strongly convex relative to $h$, 
    mirror Polyak guarantees
    \aligns{
        f(x_{t}) - f_{\star}
        \leq
        \paren{
            1 - \frac{\mu}{L+\mu}
        }^t 
        \max\braces{ 
            f(x_0) - f_{\star}, L D(x_{\star}, x_0)
        }.
    }
\end{theorem}
\removesomewhitespace
If $f$ is $\mu$-strongly convex relative to $h$, 
we might expect mirror descent to converge linearly~\citep[][Thm. 3.1]{lu2018relatively}
as the decrease in divergence~\eqref{eq:breg_stationarity} 
becomes
\aligns{
    D(x_{\star}, x_{t+1}) 
    &= 
    D(x_{\star}, x_t) + \lin{\alpha_t g(x_t), x_{\star} - x_t} + D(x_t, x_{t+1}),
    \\
    &\leq
    (1-\alpha_t \mu) D(x_{\star}, x_{t}) - \alpha_t \paren{f(x_t) - f_{\star}} + D(x_t, x_{t+1}),
    \tag*{(by strong convexity)}
    \\
    &\leq 
    (1-\alpha_t \mu) D(x_{\star}, x_{t}) - D(x_{t+1}, x_t).
    \tag*{(Using \eqref{eq:general-relative-polyak-condition})}
}
This inequality would give a linear rate
if we had a lower bound on the step-sizes \(\alpha_t\).
If $h$ is strongly convex with respect to a norm,
smoothness implies that~$L\paren{f(x_t) - f_{\star}} \geq \frac{1}{2} \norm*{g(x_t)}_*^2$
and $\alpha_t \geq \nicefrac{1}{2L}$,
but this bound does not hold in the relative setting%
~\citep[see][Lemma 3]{dragomir2021fast}.
We verify numerically
that the step-size can indeed be arbitrarily small
in Appendix~\ref{apx:pepit},
which calls for another proof technique.

\medskip

\begin{proof}
Our proof strategy is to show that 
if $\alpha_t$ is consistently large, 
we make significant progress in the divergence, as iterating the divergence contraction discussed above implies
\begin{subequations}
\begin{equation}
    D(x_{\star}, x_{T}) 
    \leq \bigl(
    \textstyle
    \prod_{t=0}^{T-1} (1 - \alpha_t \mu) 
    \bigr)
    D(x_{\star}, x_0) - D(x_{T}, x_{T-1}),
    \label{eq:rel-sc-divergence-rate}
\end{equation}
but if $\alpha_t$ is consistently small,
we make significant progress in function value,
i.e.
\alignn{
    f(x_{t+1}) - f_{\star} 
    \leq L \alpha_t \paren{f(x_t) - f_{\star}},
    \quad 
    f(x_{T}) - f_{\star} 
    \leq 
    \bigl(
        \textstyle
        \prod_{t=0}^{T-1} L \alpha_t
    \bigr)
    \paren{f(x_0) - f_{\star}}.
    \label{eq:rel-sc-function-rate}
}
\end{subequations}
Before showing how to combine those inequalities to obtain~\Cref{thm:smooth-strongly-convex},
we first prove~\cref{eq:rel-sc-function-rate}.
Combining relative smoothness 
with the fact that
$f(x_t) + \lin{g(x_t), x_{t+1} - x_t} = f_{\star}$~\eqref{eq:general-mirror-descent-step},
\alignn{
    \label{eq:opt_gap_divergence_smoothness}
    f(x_{t+1}) - f_{\star} 
    &\leq f(x_t) + \lin{g(x_t), x_{t+1} - x_t} + L D(x_{t+1}, x_t)
    - f_{\star}
    \stackrel{\smash{\eqref{eq:general-mirror-descent-step}}}{=}
    L D(x_{t+1}, x_t).
}
Using the property of the mirror Polyak step-size \eqref{eq:general-relative-polyak-condition},
\aligns{
    \alpha_t \paren{f(x_t) - f_{\star}}
    \stackrel{\smash{\eqref{eq:general-relative-polyak-condition}}}{\geq}
    \lin{\hat x_{t+1} - \hat x_t, x_{t+1} - x_t}
    =
    D(x_{t+1}, x_t) + D(x_t, x_{t+1}),
}
which implies $D(x_{t+1}, x_t) \leq \alpha_t \paren{f(x_t) - f_{\star}}$.
Combining both inequalities gives~\cref{eq:rel-sc-function-rate}.
To combine both types of progress, note that the progress in~\eqref{eq:rel-sc-divergence-rate} together with~\eqref{eq:opt_gap_divergence_smoothness} imply 
\begin{equation*}
f(x_T) - f_{\star} \leq L D(x_T, x_{T-1}) \leq L \Big(\textstyle \prod_{t=0}^{T-1} (1 - \alpha_t \mu)\Big)
    D(x_{\star}, x_0)    
\end{equation*}
Since $f(x_T) - f_{\star}$ is also bounded as in~\eqref{eq:rel-sc-function-rate},
it is bounded by the minimum of the two, so
\aligns{
    f(x_{T}) - f_{\star} 
    &\leq
    \min\Bigl\{
        \textstyle
        \paren{\prod_{t=0}^{T-1} L \alpha_t} \paren{f(x_0) - f_{\star}}, 
        L \paren{\prod_{t=0}^{T-1} \paren{1-\alpha_t\mu}} D(x_{\star}, x_0)
    \Bigr\},
    \\
    &\leq
    \rho^{\smash{T}}
    \max\braces{
        f(x_0) - f_{\star},
        L D(x_{\star}, x_0)
    },
}
where $\rho$ is the smaller of the two (geometric) average contraction rates, 
\aligns{
    \rho 
    \coloneqq
    \min\Bigl\{
        \textstyle
        \prod_{t=0}^{T-1} L \alpha_t , 
        \prod_{t=0}^{T-1} (1-\alpha_t\mu)
    \Bigr\}^{\!{\nicefrac{1}{T}}}
    \, \stackrel{\text{AM-GM}}{\leq} \,\,
    \frac{1}{T} \min \Bigl\{
        \textstyle
        \sum_{t = 0}^{T-1} L \alpha_t,
        \sum_{t = 0}^{T-1} (1 - \alpha_t \mu)
    \Bigr\}.
}
Since \(L \alpha_t\) is increasing 
and \((1 - \alpha_t \mu)\) is decreasing in $\alpha_t$, 
the worst case $\alpha_t^*$ is such that
\aligns{
    L \sum_{t = 0}^{T-1} \alpha_t^* = \sum_{t = 0}^{T-1} (1 - \alpha_t^* \mu)
    \implies \sum_{t = 0}^{T-1} \alpha_t^* = \frac{T}{L + \mu}.
}
Therefore, the average contraction rate is at most
$\rho \leq \frac{1}{T} \cdot L \cdot \frac{T}{L + \mu} = \frac{L}{L + \mu} = 1 - \frac{\mu}{ L + \mu}$.
\end{proof}

\vspace{-1em}

\subsection{Convergence rate under relative Lipschitz continuity}
\label{sec:rel-lipschitz}

If $f$ is Lipschitz continuous relative to $h$,
we get the expected~\(O(1/\sqrt{T})\)~rate.
\begin{theorem}
    \label{thm:lipschitz-convex}
    If \(f\) is \(G\)-Lipschitz relative to \(h\), 
    mirror Polyak guarantees
    \begin{equation*}
        f(\bar x_T) - f_{\star} \leq \frac{G\sqrt{2 D(x_{\star}, x_0)}}{\sqrt{T}},
        \quad \text{ for } \quad 
        \bar x_T \coloneqq \frac{1}{T} \sum_{t=0}^{T-1} x_{t}
    \end{equation*}
\end{theorem}
\removesomewhitespace
\begin{proof}
    Relative Lipschitz continuity and 
    $f(x_t) + \lin{g(x_t), x_{t+1} - x_t} = f_{\star}$ (\cref{eq:general-mirror-descent-step}) yield
    \begin{align*}
        2G^2 D(x_{t+1}, x_t) 
        \geq \lin{g(x_t), x_t - x_{t+1}}^2
        \stackrel{\eqref{eq:general-mirror-descent-step}}{=}
        (f(x_t) - f_{\star})^2.
    \end{align*}
    Combined with the contraction of divergence~\eqref{eq:breg_stationarity}, we get
    \aligns{
        D(x_{\star}, x_{t+1}) 
        \leq 
        D(x_{\star}, x_t) - \tfrac{1}{2G^2} \paren{f(x_{t}) - f_{\star}}^2.
    }
    Iterating and using Jensen's inequality yields
    (this bound also holds for the best iterate)
    \aligns{
        f(\bar x_T) - f_{\star}
        \leq \smash{\frac{1}{T}\sum_{t=0}^{T-1}} f(x_{t}) - f_{\star}
        \leq \sqrt{\textstyle\frac{1}{T} \sum_{t=0}^{T-1} \paren{f(x_{t}) - f_{\star}}^2}
        \leq \frac{G\sqrt{2 D(x_{\star}, x_0)}}{\sqrt{T}}.
        \tag*{\jmlrQED}
    }
    \noqed
\end{proof}

\vspace{-1em}

\subsection{Relative Lipschitz continuity and strong convexity}
\label{sec:rel-lipschitz_str_cvx}

In the Euclidean case, 
the Polyak step-size achieves an \(O(1/T)\) convergence rate 
for the best iterate or for a tail-averaged iterate~\citep{hazan2019revisiting}.
We show a similar, although slightly weaker result 
in the relative setting: an \(O(\log T/T)\) rate for the best iterate.

\pagebreak[2]

\begin{theorem}
    \label{thm:lipschitz-strongly-convex}
    If \(f\) is \(G\)-Lipschitz continuous and $\mu$-strongly convex relative to~\(h\) on a set~$\setC$, 
    mirror Polyak guarantees that for $T > 1$,
    \aligns{
        \min_{t\leq T} f(x_t) - f_{\star}
        \leq \smash{\frac{4G^2}{\mu} \frac{\log(T)}{T}}.
    }
\end{theorem}
\removesomewhitespace
Combining Lipschitz continuity and strong convexity requires care, 
as both assumptions cannot hold simultaneously on unbounded domains
even in the Euclidean case.
The proof of \Cref{thm:lipschitz-strongly-convex}
relies on the following bounds on the constraint set $\setC$ 
that are implied by the assumptions of relative Lipschitz continuity and strong convexity. 
\begin{lemma}
    \label{lem:lipschitz-strongly-convex-properties}
    \begin{subequations}%
        If $f$ is $G$-Lipschitz continuous relative to $h$
        on a set $\setC$,
        mirror Polyak satisfies
        \alignn{
            \alpha_t
            &\geq
            \frac{1}{2G^2}\paren{f(x_t) - f_{\star}} \quad \text{and}
            \label{eq:step-size-lower-bound-lipschitz}
            \\
            f(x) - f_{\star} 
            &\leq
            G\sqrt{2D(x_{\star}, x)}
            \qquad \text{for all}~x \in \setC.
            \label{eq:optgap-bounded-by-distance}
        }
        Additionally, if $f$ is also $\mu$-strongly convex relative to $h$ on $\setC$, then
        \begin{equation}
            D(x_{\star}, x) 
            \leq \frac{2G^2}{\mu^2} 
            \qquad 
            \text{for all}~x \in \setC.
            \label{eq:diameter-bounded}
        \end{equation}
    \end{subequations}
\end{lemma}
\removesomewhitespace
\begin{proof}
    For~\cref{eq:step-size-lower-bound-lipschitz},
    using the definition 
    of the mirror Polyak step-size~\eqref{eq:general-relative-polyak-condition}, 
    we have 
    \aligns{
        \alpha_t
        &\geq 
        \frac{\lin{\hat x_t - \hat x_{t+1}, x_t - x_{t+1}}}{f(x_t) - f_{\star}}
        = 
        \frac{D(x_{t+1}, x_t) + D(x_t, x_{t+1})}{f(x_t) - f_{\star}}
        \geq
        \frac{D(x_{t+1}, x_t)}{f(x_t) - f_{\star}}.
    }
    We get ~\cref{eq:step-size-lower-bound-lipschitz} 
    by relative Lipschitz continuity
    and the mirror Polyak step-size~\eqref{eq:general-mirror-descent-step},~as
    \alignn{
        \label{lem:bounded-distance}
        D(x_{t+1}, x_t) \geq
        \frac{1}{2G^2} \lin{g(x_t), x_{t+1} - x_t}^2
        \stackrel{\smash{\eqref{eq:general-mirror-descent-step}}}{=}
        \frac{1}{2G^2} \paren{f(x_t) - f_{\star}}^2,
    }
    For \cref{eq:optgap-bounded-by-distance,eq:diameter-bounded},
    by relative Lipschitz continuity and strong convexity,
    \aligns{
        G\sqrt{2D(x_{\star}, x)}
        \geq \lin{g(x), x - x_{\star}}
        \geq f(x) - f(x_{\star}) + \mu D(x_{\star}, x)
        \geq 0.
    }
    This gives~\cref{eq:optgap-bounded-by-distance} as $D(x_{\star}, x) \geq 0$, 
    and~\cref{eq:diameter-bounded} as $f(x) - f(x_{\star}) \geq 0$.
\end{proof}

\begin{proof} \textbf{of \cref{thm:lipschitz-strongly-convex}}
    For conciseness, we define
    $d_t \coloneqq D(x_{\star}, x_t)$ and $\gamma_t \coloneqq f(x_t) - f_{\star}$.
    Using strong convexity as in \Cref{thm:smooth-strongly-convex}
    and the lower bound on the step-size~\eqref{eq:step-size-lower-bound-lipschitz} yields
    \aligns{
        d_{t+1} 
        &
        \leq
        (1 - \alpha_t \mu) d_t - D(x_{t+1}, x_t)
        <
        (1 - \alpha_t \mu) d_t
        \stackrel{
            \smash{
                \eqref{eq:step-size-lower-bound-lipschitz}
            }
        }{\leq} 
        \,
        \paren{1 - \frac{\mu}{2G^2} \gamma_t} d_t.
    }
    Iterating the contraction 
    and using that $1-x \leq e^{-x}$,
    we have 
    \aligns{
        d_{T} 
        \leq
        \prod_{t=0}^{T-1} \paren{1 - \frac{\mu}{2G^2} \gamma_t} d_0
        \leq
        \exp\paren{\textstyle
            - \frac{\mu}{2G^2} \sum_{t=0}^{T-1} \gamma_t
        } d_0.
    }
    We can then proceed by contradiction. 
    Suppose that all the optimality gaps are large, that is, 
    assume~$\gamma_t \geq (\nicefrac{4G^2}{\mu})(\nicefrac{\log(T)}{T})$ 
    for all $t \leq T$.
    Then, the bound on the divergence \(d_T\) yields 
    \aligns{
        d_{T} 
        \leq
        \exp\paren{- 2 \log(T)} d_0
        = 
        \frac{d_0}{T^2}.
    }
    To bound the optimality gap, 
    we use that $\gamma_T \leq G\sqrt{2 d_T}$~\eqref{eq:optgap-bounded-by-distance}
    and 
    $\sqrt{d_0} \leq \nicefrac{\sqrt{2}G}{\mu}$~\eqref{eq:diameter-bounded}
    to~get
    \aligns{
        \gamma_T 
        \leq 
        G\sqrt{2 d_T}
        \leq 
        G\sqrt{\frac{2 d_0}{T^2}}
        \leq 
        \frac{2G^2}{\mu}\frac{1}{T}.
    }
    To be consistent with our assumed lower bound 
    $\gamma_t \geq (\nicefrac{4G^2}{\mu})(\nicefrac{\log(T)}{T})$
    we must have that 
    \aligns{
        \frac{4G^2}{\mu} \frac{\log(T)}{T} 
        \leq \gamma_T \leq 
        \frac{2G^2}{\mu}\frac{1}{T},
        \quad \text{ which implies } \quad
        T \leq e^{1/2}.
    } 
    For $T \geq 2$, we have a contradiction 
    and at least one $\gamma_t$ must be smaller than \smash{$\frac{4G^2}{\mu} \frac{\log(T)}{T}$}.
\end{proof}

In the Euclidean case, we expect an $O(\log(T)/T)$ rate for the average iterate,
while the best iterate should converge at $O(1/T)$.
This distinction also holds for mirror descent 
with decreasing step-sizes~\citep{lu2019relativecontinuity}.
While it might be possible to improve the rate, 
it would require more refined inequalities 
than the ones in Lemma~\ref{lem:lipschitz-strongly-convex-properties},
as they can be satisfied by~%
$\gamma_t = \slashfrac{G^2\log(T)}{\mu T} $,
$d_{0} = \slashfrac{G^2}{\mu^2}$, 
and~$d_{t+1} = (1 - \slashfrac{\gamma_t\mu}{2G^2}) d_t$ 
for~large~$T$.

\section{Polyak step-size without knowing the optimal value}
\label{sec:lifting}

Now that we have established the convergence rates 
of mirror Polyak, we move on to the problem 
of using it when the minimum value $f_{\star}$ is unknown.
Many methods have been proposed to estimate $f_{\star}$ during optimization:
using decreasing upper bounds~\citep{held1974validationos}, 
the current objective value~\citep{bazaraa1981ontc},
or more complex target value algorithms~\citep{kim1990variable},
and many variants of these schemes have been 
proposed over the years~
\citep{%
brannlund1993relaxation,%
kiwiel1999efficiency,%
goffin1999convergence,%
sherali2000variable,%
nedic2001convergence,%
fumero2001modified,%
lorenz2014infeasible%
}.
But none preserves the convergence rate of the Polyak step-size. 
Some heuristics 
converge asymptotically but do not have anytime rates~\citep[e.g.][]{brannlund1993relaxation},
while restart schemes require fixing the number of iterations $T$ in advance,
or paying a $\log(T)$ factor with a doubling trick~\citep{hazan2019revisiting}.
To bypass the need to know $f_\star$,
we propose an alternative formulation.

Suppose that, in addition to our primal objective $f$, 
we have a dual such that strong duality holds,
that is, a convex function $q$ such that
$\min_x f(x) = \max_y -q(y)$.
Then, instead of minimizing $f$ directly,
we can optimize the~lifted~problem
\alignn{
    \label{eq:lifting_def}
    \min_{x, y} F(x, y) \quad \text{ where }  \quad F(x, y) \coloneqq f(x) + q(y).
}
The lifted problem $F$ is convex in the joint parameters $(x,y)$, 
has a minimum value of~$0$, 
and satisfies $F(x_t, y_t) = f(x_t) - f_{\star} - \inf_y q(y) + q(y_t) \geq f(x_t) - f_{\star}$
so that guarantees on~$F(x_t, y_t)$ 
transfer to~$f(x_t) - f_{\star}$.
On the example problems we consider,
the lifted problem is smooth and strongly convex relative to a reference function~$h$
with the same condition number as the primal problem
under a careful choice 
of reference function~$h(x, y) = \frac{1}{2}\norm{x}^2 + h_{\setY}(y)$.
We can thus apply mirror Polyak to the lifted problem
while preserving the worst-case guarantee of the Polyak step-size.
To be feasible, this approach requires a dual that can be evaluated as efficiently as the primal.
We give examples for a common class of regularized data-fitting problems 
which can be evaluated efficiently
using the same dual problem as SDCA~\citep{shalev2013stochastic}. We summarize the duality results and properties of the lifted problem in the next proposition.
\begin{proposition}%
    \label{prop:fenchel-dual}
    Let \(\phi\) and \(r\) be closed convex functions on \(\R^n\) and \(\R^d\), respectively, 
    with conjugates~$\phi^*$ and~$r^*$.
    Let $A \in \R^{n \times d}$, 
    \(\lambda > 0\) and the functions \(f\)~and~\(q\)~be
    \begin{equation*}
        f(x) \coloneqq \phi(Ax) + \lambda r(x) 
        \quad \text{and} \quad
        q(y) \coloneqq \phi^*(y) 
        + \lambda r^*\paren{-\frac{1}{\lambda} A^T y} 
        \qquad \forall x \in \R^d, \forall y \in \R^n.
    \end{equation*}
    Strong duality holds and $\min_x f(x) = \max_y -q(y)$ if
    there exists an $x$ in the relative interior of $\dom r$ 
    such that~$Ax$ is in the relative interior of $\dom \phi$~\citep[][Cor. 31.2.1]{rockafellar1970convex}. 
\end{proposition}

The lifting comes at the cost of an increased dimensionality
as we now have to optimize in both the $x$ and $y$ variables.
It is also not a black-box method 
since evaluating the lifting $F$ as in~\eqref{eq:lifting_def} requires more than a gradient oracle of $f$.
Nonetheless, $F$ has an optimum of $0$, enabling the use of the Polyak step-size.
Although $F$ may not have the same 
Euclidean smoothness and strong convexity constants as $f$,
making classical Polyak potentially slower,
we can in some cases preserve those 
properties relative to a proper choice of reference function
for mirror Polyak. 
The next proposition gives such a mirror map for 
smooth and strongly convex problems 
with~Euclidean~regularization.

\begin{proposition}
    \label{prop:lifting-smooth-strongly-convex}
    Let $f, q, \phi, r$ and $A$ be as in Proposition~\ref{prop:fenchel-dual}
    and define $L \coloneqq \lambda + \beta \norm*{A}_2^2$.
    If~\(r = \frac{1}{2}\norm*{\cdot}_2^2\) and~\(\phi\) is~\(\beta\)-smooth,
    then $f$ is $L$-smooth 
    and \(\lambda\)-strongly convex in Euclidean norm,
    and the lifted problem~$F(x,y) \coloneqq f(x) + q(y)$ 
    is also $L$-smooth and $\lambda$-strongly convex
    relative to the reference function
    $h(x,y) \coloneqq \frac{1}{2}\norm*{x}_2^2 + \frac{1}{\lambda} \phi^*(y)$.
\end{proposition}

Using this reference function, 
mirror Polyak on the lifted problem~$F$
preserves the worst-case guarantee 
that the idealized Polyak step-size knowing $f_{\star}$ 
would have on the primal problem~$f$.
We give two examples, 
first on regularized least-squares
where the lifting preserves Euclidean smoothness,
and on regularized logistic regression
where the dual is not smooth in any norm
but well-behaved relative to the~negative~entropy.

\begin{example}[Regularized least-squares]
    \label{ex:linear-regression}
    Consider the regularized least-squares objective
    with data~$A \in \R^{n \times d}$ and $b~\in \R^n$,
    primal variables~$x \in \R^d$ 
    and dual variables~$y \in \R^n$,
    \aligns{
        f(x) \coloneqq \tfrac{1}{2} \norm{Ax - b}^2 
        + \tfrac{1}{2} \norm{x}^2, 
        &&
        q(y) = \frac{1}{2} \norm{y}^2 + \lin{y, b} + \frac{1}{2} \big\lVert -A^T y\big\rVert^2.
    }
    The primal is $L \coloneqq (\norm{A}_2^2 + 1)$-smooth 
    and~\(1\)-strongly convex,
    but the optimal value~$f_{\star} \geq 0$ is unknown.
    Since $q$ is \(L\)-smooth and \(1\)-strongly convex, 
    $F$ %
    has the same condition number~as~$f$.
\end{example}

\begin{example}[Logistic regression]
    \label{ex:logistic-regression}
    Consider the regularized logistic regression objective
    with data $A \in \R^{n \times d}$ and $b \in \{-1,1\}^n$,
    primal variables $x \in \R^d$ 
    and dual variables $y \in [0,1]^n$:
    \aligns{
        f(x) &= \phi(b \odot Ax) + \lambda \tfrac{1}{2} \norm{x}^2,
        &
        q(y) &\coloneqq \phi^*(y) + \lambda \tfrac{1}{2} \norm{-\tfrac{1}{\lambda} A^T (b \odot y)}^2,
        \\
        \phi(z) &\coloneqq \sum_{i=1}^n \log(1 + \exp(z_i)),
        &
        \phi^*(y) &= \sum_{i=1}^n \big(y_i \log y_i + (1 - y_i) \log(1 - y_i)\big),
    }
    where~\(\odot\) denotes element-wise multiplication.
    The primal is~$L \coloneqq \paren*{ \norm{A}_2^2/4 + \lambda}$-smooth 
    and~\(\lambda\)-strongly convex,
    but $q$ is neither smooth nor Lipschitz in any norm,
    as~$\nabla \phi^*(y)$ diverges when any~$y_i$ approaches $0$ or $1$.
    But it is well-behaved relative to the negative entropy~$\phi^*(y)$;
    the lifted problem
    $F$ is $L$-smooth and $\lambda$-strongly convex relative 
    to~$h(x,y) = \frac{1}{2}\norm{x}^2 + \frac{1}{\lambda} \phi^*(y)$.

\end{example}

\section{Concluding remarks}
\label{sec:experiments}

We have provided an analysis of mirror Polyak,
showing that it adapts to the relative geometry of the problem 
like its Euclidean counterpart, 
and proposed a lifted problem based on a duality gap
to bypass the need to know the optimal value.
To conclude, we illustrate the theoretical results
on regularized linear and logistic regression problems.
Experiment details can be found in Appendix~\ref{apx:experiments_details}.
We compare 
mirror Polyak on the lifted problem
with methods that estimate~\(f_{\star}\);
a level estimation method~\citep{brannlund1993relaxation}
with and without hyperparameter tuning
and TwinPolyak \citep{abdukhakimov2025polyakstepsizeestimatingoptimal},
as well as an idealized Polyak step-size knowing~\(f_{\star}\).
\begin{figure}[t]
\centering
\includegraphics[width=\textwidth]{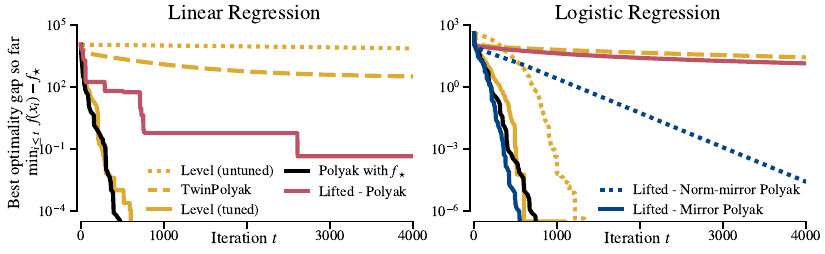}
\caption{
    Mirror Polyak on the duality gap converges without 
    having to know~\(f_{\star}\).
    Illustration on regularized linear 
    and logistic regression problems.
    On the logistic regression problem, 
    mirror Polyak outperforms the norm-based variant 
    as solving the projection leads to larger steps~(Prop.~\ref{prop:relative-polyak-larger}).
    Methods that estimate $f_\star$ perform well if tuned,
    but no setting works for all problems.
}
\label{fig:lifted_duality_gap_experiment}
\end{figure}

The results, shown in \Cref{fig:lifted_duality_gap_experiment},
illustrate the behaviors described earlier.
Mirror Polyak on the gap is fully specified by the problem structure
and choice of primal-dual gap,
while the performance of level estimation methods 
depends heavily on having their additional hyperparameters tuned.
We observe that mirror Polyak on the lifted problem
does not necessarily improve the primal gap at each step,
as it could only improve the dual gap.
In the logistic regression case,
as the reference is strongly convex in Euclidean norm,
the norm-based variant of mirror Polyak~\eqref{eq:norm-based-polyak}
is well-defined but makes less progress 
as it takes smaller steps than mirror 
Polyak~(see Proposition~\ref{prop:relative-polyak-larger}).

\paragraph{Stochastic extensions.}
Extensions of these ideas to the stochastic setting
remain open and of interest for applications in machine learning.
The Polyak step-size 
has been applied successfully 
for finite-sum problems under interpolation,
when all functions share a minimizer
and an optimal value of 0,
but is difficult to apply without interpolation~\citep{orvieto2022dynamics}.
Existing results assuming interpolation
or using the idealized variant of~\citet{gower2025analysis}
would show mirror Polyak converges asymptotically 
as it leads to a contraction~\citep[][Prop.~4.13]{bauschke2003bregman}, 
but its non-asymptotic behavior is unknown.
Our lifting technique 
might help improve existing schemes in the stochastic case~\citep{garrigos2023functionvaluelearning,jiang2023adaptive}.

\clearpage
\section*{Acknowledgements}
Frederik Kunstner is supported by a Marie Skłodowska-Curie Fellowship from the European Union's Horizon Europe Research and Innovation program under Grant Agreement No. 101210427.
Ryan D'Orazio's work is funded by Ioannis Mitliagkas' CIFAR chair.
Victor S. Portella's work was funded,
in part, by the São Paulo Research Foundation (FAPESP), Brazil, process number 2024/09381-1.
Adrien Taylor is supported by the European Union (ERC grant CASPER 101162889). The
French government also partly funded this work under the management of Agence Nationale de la
Recherche as part of the ``France 2030'' program, references ANR-23-IACL-0008 ``PR[AI]RIE-PSAI''.

\bibliography{sn-bibliography}%

\appendix

\begin{appendices}

\clearpage
\section{Constraint qualification for constrained mirror Polyak}
\label{apx:slater-condition}

\paragraph{Constraint qualification.}  
Let us show that the problem~\eqref{eq:constrained_polyak_def} satisfies constraint qualifications. That is, define the half-space $\cH_t$ by
\begin{equation*}
    \cH_t \coloneqq \{x \in \R^d \colon l_t(x) \coloneqq f(x_t) + \lin{g(x_t), x - x_t} \leq f_{\star}\}.
\end{equation*}
The problem in~\cref{eq:constrained_polyak_def} is a Bregman projection of $x_t$ onto the boundary of $\cH_t$ (intersected with $\cX \cap \cC$). Since $x_t \not\in \cH_t$ (assuming $x_t$ is not a minimum of $f$), one may verify that the projection onto the half-space $\cH_t$ or its boundary are equivalent. 
Moreover, since the above set is polyhedral, we conclude that Slater's condition for the problem in~\cref{eq:constrained_polyak_def} is
\begin{equation*}
    \cH_t \cap \relint(\cC \cap \cX) \neq \emptyset,
\end{equation*}
By assumption $\relint \cC \cap \interior \setX \neq \emptyset$, and thus $\relint (\cC \cap \setX) = \relint \cC \cap \interior \setX$ \citep[Theorem~6.5]{rockafellar1970convex}. Therefore, the above constraint qualification conditions  become
\begin{equation*}
    \cH_t \cap \relint(\cC) \cap \interior(\cX) \neq \emptyset,
\end{equation*}
To see that the above holds, simply  note that by assumption $x_{\star} \in \relint(\cC) \cap \interior \setX$, and $x_{\star} \in \cH_t$ is true by convexity of $f$ (or, in other words, the boundary of $\cH_t$ is a separating hyperplane of $x_t$ and $x_{\star}$).

\paragraph{Main assumption and relaxations.}  
The assumption that $x_{\star} \in \relint(\setC) \cap \interior \setX$ in Lemma~\ref{prop:breg_stationarity} is to guarantee constraint qualification for~\cref{eq:constrained_polyak_def}.
    Specifically, we use it to show that the intersection 
    $\relint \setC \cap \interior  \cX \cap \setH_t$ is non-empty for every non-optimal iterate $x_t$. 
Alternatively, we could assume that $f$ is strictly convex, 
or that $\relint \setC \cap \interior  \cX \cap \setH_t \neq \emptyset$ for all $t$ as does~\citet{kiwiel1997bregman},
or use an upper bound on~$f_{\star}$ in the definition of the hyperplane
instead of the exact optimal value~\citep{kiwiel1998generalizedbp},
although this last option prevents convergence to the minimum.

This assumption guards against the possibility that the solution 
of the projection is at $\alpha_t \to \infty$.
Suppose the domain of $f$, the set $\setX$, is the nonnegative orthant and \(\cC\) is the simplex, thus both are closed and convex. Further assume that $f$
is minimized at a unique point~$x_{\star}$ in $\cC$ and it is on the boundary of $\setX$~(the optimal probability vector has one or more coordinates with no mass).
If the hyperplane selected passes through $x_\star$ but does not intersect $\relint(\cC) \cap \interior(\cX) $ (which is just $\relint(\cC)$ in this case), 
the only solution for the Bregman projection is on the boundary of~$\setX$.
For entropy-type mirror maps, such boundary points correspond to infinite
values in the dual coordinates, so the projection may only be realized
in the limit $\alpha_t\to\infty$.
Handling such limiting projections requires additional care.

\clearpage
\section{Performance estimation verification for small steps}
\label{apx:pepit}
We used \emph{performance estimation problems} to verify the tightness of certain results.
In particular, it might appear strange at first sight 
that the mirror Polyak step-size can be arbitrarily small in \Cref{thm:smooth-strongly-convex}, 
and that the proof requires combining two types of Lyapunov analyses (namely convergence in Bregman divergence and convergence in function values). 
\Cref{fig:pepit_valid} illustrates that arbitrarily small step-sizes are possible  
and that the two (simple) bounds are (arguably unfortunately) tight.

For each tested value of $\alpha$, 
the program verifies that there exists a pair $(f,h)$ 
such that the problem instance~$f$ 
is $L$-smooth and $\mu$-strongly convex 
relative to the reference function~$h$,
for which mirror descent with step-size $\alpha$ is valid. 
Among those feasible instances, 
the program picks one that maximizes the one-step ratios
\aligns{
    \frac{D(x_\star, x_{1})}{D(x_\star, x_0)} \leq (1-\alpha \mu),
    &&
    \frac{f(x_{1}) - f_\star}{f(x_0) - f_\star} 
    \leq L \alpha,
}
over $x_0, f, h : 0 \preceq \mu\nabla^2 h(x) \preceq \nabla^2 f(x) \preceq L\nabla^2 h(x) \forall x$, 
with $x_{1}$ given by mirror Polyak with step-size $\alpha$. The detailed techniques for formulating 
and solving such problems are provided 
by~\citet{barre2020complexity} and \citet{dragomir2022optimal}, 
both of which are implemented in the PEPit package~\citep{goujaud2024pepit}, 
which we use for simplicity. 

\vfill

\begin{figure}[h]
    \includegraphics[width=\linewidth]{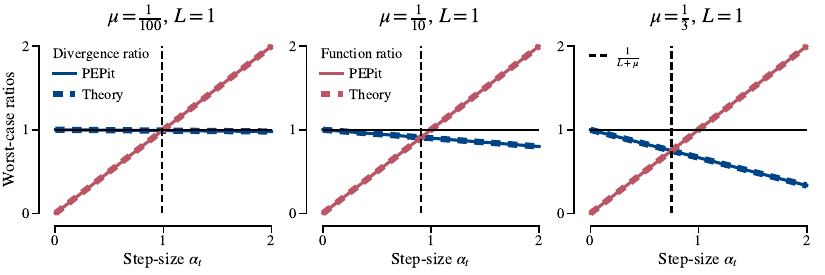}
    \centering
    \caption{
        Comparison of \emph{tight} one-iteration analyses 
        using PEPit~\citep{goujaud2024pepit} 
        with the bounds~\eqref{eq:rel-sc-divergence-rate} 
        and~\eqref{eq:rel-sc-function-rate} 
        in the proof of~\Cref{thm:smooth-strongly-convex}. 
    Tests with other values of $L$ behave similarly 
    within the range of feasible step-sizes, 
    $\alpha_t\in\big(0,\tfrac{1}{\mu}\big)$.}
    \label{fig:pepit_valid}
\end{figure}

\vfill
\null

\clearpage
\section{Experimental details}
\label{apx:experiments_details}

\paragraph{Dataset and problem formulation.}
The logistic regression experiments in \Cref{fig:lifted_duality_gap_experiment}
were run on the a1a~\citep{becker1996adult} 
distributed by the LIBSVM dataset repository~\citep{chang2011libsvm}
using the transformation described by \citet{platt1998fast}.
The linear regression uses the abalone~\citep{abalone1994} dataset.
The input data is normalized (subtract the mean, divide by standard deviation).
In both cases, the regularization parameter is set to $\lambda = 1$ 
(in the sum formulation of the problem, otherwise $\lambda = 1/n$ for the average).

\paragraph{Initialization.}
All methods are initialized at 0, 
except for Twin Polyak,
for which we use the initialization $x_0 = 0$ and $y_0 \sim \Normal{0, 1}$.
\citet{abdukhakimov2025polyakstepsizeestimatingoptimal}
recommend using $x_0, y_0 \sim \Normal{0,1}$.
Changing $x_0 = 0$ does not change the ordering of the methods,
so we use it to ensure all methods start at the same function value.
The dual variables for the logistic regression problem are initialized at $y_0 = 1/2$.

\paragraph{Getting $f_{\star}$.}
The minimum objective value $\tilde f_\star$ used to normalize the curves 
and set the idealized Polyak step-size 
is obtained by BFGS run until $\norm{\nabla f(x)} \leq 10^{-6}$,
which implies the optimality gap is accurate up to
$\tilde f_{\star} - f_{\star} \leq 10^{-12}$,
as $f$ is 1-strongly convex.

\paragraph{Mirror map for the dual.}
For the logistic regression problem,
we use the Euclidean mirror map for the primal variables,
$h_\setX(x) = \frac{1}{2} \norm{x}^2$,
and the negative entropy mirror map for the dual variables,
$h_\setY(y) = \sum_{i=1}^n y_i \log y_i + (1 - y_i) \log(1 - y_i)$.
For the norm-based variant of mirror Polyak, 
we use $h'_\setY(y) = \frac{1}{4} h_\setY(y)$
to ensure it is 1-strongly convex.

\paragraph{Hyperparameters.}
The level estimation method of \citet{brannlund1993relaxation}
(also analyzed by \citealt{goffin1999convergence,nedic2001convergence}),
requires two parameters, $\delta$ and $R$.
\citet{goffin1999convergence}
discuss that $\delta$ should estimate the initial optimality gap,~$\delta \approx f(x_0) - f_{\star}$,
while $R$ should estimate the distance to the optimum,~$R \approx \norm{x_0 - x_{\star}}$.
The parameters need not be set exactly, 
as the method converges asymptotically for any positive values,
but the performance of the method depends on getting the order of magnitude right.
The values used in \Cref{fig:lifted_duality_gap_experiment}
are $\delta = f(x_0)$ and $R = 1$ for the tuned method
and $\delta = 1, R = 1$ for the untuned method. 
Much smaller or much larger values of $R$ gave poorer performance.

\end{appendices}

\end{document}